\documentclass[a4paper,11pt]{article}
\usepackage[latin1]{inputenc}
 \usepackage[T1]{fontenc}
  \usepackage{amsmath,amsthm,amssymb}
 \usepackage[all]{xy}
\def \N{{\mathbb N}}

\def \R{{\mathbb R}}
\def \Z{{\mathbb Z}}
\def \C{{\mathbb C}}
\def \K{{\mathbb K}}
\def \1{{\mathbb 1}}
\newtheorem{Exemp}{Examples}
\newtheorem{Thm}{Theorem}
\newtheorem{Prop}{Proposition}
\newtheorem{Def}{Definition}
\newtheorem{Rem}{Remark}
\newtheorem{Lem}{Lemma}

\newtheorem{Def Nota}{Definitions and notations} 
\newtheorem{Cor}{Corollary}

\font\ninerm=cmr9
\long\outer\def\abstract#1{\bigskip\vbox{\noindent\ninerm
\baselineskip=10pt#1}\nobreak\bigskip}
\catcode`\â= \active \def â{\^a}
\catcode`\ê= \active \def ê{\^e}
\catcode`\ô= \active \def ô{\^o}
\catcode`\û= \active \def û{\^u}
\catcode`\î= \active \def î{{\^\i}}
\catcode`\é= \active \def é{\'e}
\catcode`\è= \active \def è{\`e}
\catcode`\à= \active \def à{\`a}
\catcode`\ù= \active \def ù{\`u}
\catcode`\ç= \active \def ç{\c {c}}
\catcode`\ï= \active \def ï{{\"\i}}

\newcount\numero
\def\exo#1{\advance\numero by 1\bigskip
{\noindent\tenbf #1\the\numero. }}
\def\frac#1#2{{#1\over #2}}
\title{Any law of group metric invariant is an inf-convolution.}   
\author{Mohammed Bachir}
\begin{document}
%\newpage 
\maketitle
\begin{center} {\it Laboratoire SAMM 4543, Université Paris 1 Panthéon-Sorbonne, Centre P.M.F. 90 rue Tolbiac 75634 Paris cedex 13}
\end{center}
\begin{center} 
{\it Email : Mohammed.Bachir@univ-paris1.fr}
\end{center}
\noindent\textbf{Abstract.} In this article, we bring a new light on the concept of the inf-convolution operation $\oplus$ and provides additional informations to the work started in \cite{Ba1} and \cite{Ba2}. It is shown that any internal law of group metric invariant (even quasigroup) can be considered as an inf-convolution. Consequently, the operation of the inf-convolution of functions on a group metric invariant is in reality an extension of the internal law of $X$ to  spaces of functions on $X$. We give an example of monoid $(S(X),\oplus)$ for the inf-convolution structure, (which is dense in the set of all $1$-Lipschitz bounded from bellow functions) for which, the map $\arg\min : (S(X),\oplus) \rightarrow (X,.)$ is a (single valued) monoid morphism. It is also proved that, given a group complete metric invariant $(X,d)$, the complete metric space $(\mathcal{K}(X),d_{\infty})$ of all Katetov maps from $X$ to $\R$ equiped with the inf-convolution has a natural monoid structure which provides the following fact: the group of all isometric automorphisms $Aut_{Iso}(\mathcal{K}(X))$ of the monoid $\mathcal{K}(X)$, is isomorphic to the group of all isometric automorphisms $Aut_{Iso}(X)$ of the group $X$. On the other hand, we prove that the subset $\mathcal{K}_C(X)$ of $\mathcal{K}(X)$ of convex functions on a Banach space $X$, can be endowed with a convex cone structure in which $X$ embeds isometrically as Banach space.\\
\\
{\bf Keyword, phrase:} Inf-convolution; group and monoid structure; Katetov functions.\\
{\bf 2010 Mathematics Subject:} 46T99; 26E99; 20M32.
\tableofcontents
%\newpage
\section{Introduction.} 
This article brings some additional informations to the study of the inf-convolution structure developed in \cite{Ba1} and \cite{Ba2}. Given a set $X$, a map $\alpha: X\times X \rightarrow X$ and two real valued functions $f$ and $g$ defined on $X$. The inf-convolution of $f$ and $g$ with respect the map $\alpha$ is defined as follows 
\begin{eqnarray}  \label{Al}
f\underbrace{\oplus}_{\alpha} g (x):= \inf_{y,z\in X/\alpha(y,z)=x}\left\{f(y)+g(z)\right\}; \forall x\in X
\end{eqnarray}
with the convention that $f\underbrace{\oplus}_{\alpha} g (x)=+\infty$ if $\left\{y,z\in X/\alpha(y,z)=x\right\}= \emptyset$.\\

Historically, the inf-convolution appeared as a tool of functional analysis and optimization and starts with the works of Mac Shane \cite{MS}, Fenchel, Moreau and Rockafellar; see \cite{Rok} for references. See also the book of J.-B. Hiriart-Urruty and C. Lemarechal \cite{HL} and the survey of T. Strömberg \cite{TS} for various properties of the inf-convolution operation. We proved in \cite{Ba1} and \cite{Ba2}, that the inf-convolution also enjoys remarkable algebraic properties. For example, we proved that the set $(Lip^1_+(X),\oplus)$ of all no negative and $1$-Lipschitz functions defined on a complete metric invariant group $(X,d)$, is a monoid and its group of unit is isometrically isomorphic to $X$. This result means that the monoid structure of $(Lip^1_+(X),\oplus)$ completely determines the group structure of $X$ whenever $X$ is an group metric invariant.\\

 In this paper, we give additional lighting to the understanding of the inf-convolution operation. Indeed, it seems that the inf-convolution is not an ''external'' operation to the space $X$ acting on it, but is in reality a canonical extension of the internal law of $X$ to the space $Lip^1_+(X)$, whenever $X$ is a group metric invariant. In oder words, any internal law of metric invariant group (even quasigroup) is an inf-convolution. This approach is motivated by Proposition \ref{int} and Theorem \ref{int2} below.\\
 
  A metric space $(X,.,d)$ equipped with an internal law $. : (y,z)\mapsto y. z$ defined from $X\times X$ into $X$ is said to be metric invariant, if 
\begin{eqnarray} 
d(x. y,x. z)=d(y. x,z. x)= d(y,z)\hspace{2mm}\forall x,y, z\in X.\nonumber
\end{eqnarray}
Note that every group is metric invariant for the discreet metric. For examples of not trivial group metric invariant, see \cite{Ba2} (For informations on group complete metric invariant see \cite{K}). Let us denote by $\gamma : x\in X \mapsto \delta_x$ the Kuratowski operator, where $\delta_x: t\in X\mapsto d(x,t).$ We denote by $\hat{X}$ the image of $X$ under the Kuratowski operator, $\hat{X}:=\gamma(X)$. 
% i.e for all $a,b, x\in X$
%\begin{eqnarray}
%\gamma(a)\oplus \gamma(b)(x)&:=& \inf_{yz=x:\hspace{2mm} y, z\in X}\left\{\gamma(a)(y)+\gamma(b)(z)\right\}. \nonumber\\
 %                           &:=& \inf_{yz=x:\hspace{2mm} y, z\in X}\left\{d(a,y)+d(b,z)\right\}. \nonumber
%\end{eqnarray}
The set $\hat{X}$ is endowed with the sup-metric $$d_{\infty}(\gamma(a),\gamma(b)):=\sup_{x\in X}|\gamma(a)(x)-\gamma(b)(x)|.$$ It is well known and easy to see that the  Kuratowski operator $\gamma$ is an isometry: for all $a,b\in X$ 
$$d_{\infty}(\gamma(a),\gamma(b))=d(a,b).$$  We define the inf-convolution on $\hat{X}$ as in the formula (\ref{Al}). For two element $\gamma(a),\gamma(b)\in \hat{X}$,
\begin{eqnarray}  
\left(\gamma(a)\oplus \gamma(b)\right)(x):=\inf_{y, z\in X/y. z=x }\left\{\gamma(a)(y)+\gamma(b)(z)\right\}.\nonumber
\end{eqnarray}

We obtain the following result which say that $(X,.)$ and $(\hat{X}, \oplus)$ has in general the same algebraic structure. Recall that a quasigroup is a nonempty magma $(X , . )$ such that for each pair $(a, b)$ the equation $a.x = b$ has a unique solution on $x$ and the equation $ y.a = b$ has a unique solution on $y$. A loop is an quasigroup with an identity element and a group is an associative loop. 
%%%%%%%%%%%%%%%%%%%%%%%%%%%%%%%%%%%
%%%%%%%%%%%%%%%%%%%%%%%%%%%%%%%%%%%%
%%%%%%%%%%%%%%%%%%%%%%%%%%%%%%%%%%%
\begin{Prop} \label{int} Let $(X,.,d)$ be a metric invariant space. Then, the following assertions are equivalent.
\item $(1)$ $(X,.)$ is a quasigroup (respectively, loop, group, commutative group)
\item $(2)$ $(\hat{X}, \oplus)$ is a quasigroup (respectively, loop, group, commutative group).\\

In this case, the Kuratowski operator $\gamma: (X,.,d)\rightarrow (\hat{X},\oplus, d_{\infty})$ is an isometric isomorphism of quasigroups (respectively, loops, groups, commutative groups). 
\end{Prop}

 We then ask whether the operation $\oplus$ of $(\hat{X}, \oplus)$ naturally extends to the whole space $(Lip^1_+(X),\oplus)$. An answer is given by the following result. The part $(1) \Rightarrow (2)$ was established in \cite{Ba1} for Banach spaces in convex setting and in \cite{Ba2} in the group framework (as well as the description of the group of unit of $(Lip^1_+(X),\oplus)$). 
\begin{Thm} \label{int2} Let $(X,.,d)$ be a complete metric invariant quasigroup.  Then the following assertions are equivalent.
\item $(1)$ $(X,.)$ is a (commutative) group. 
\item $(2)$ $(Lip^1_+(X),\oplus)$ is a (commutative) monoid.\\

In this case, the identity element of  $(Lip^1_+(X),\oplus)$ is $\gamma(e)$ where $e$ is the identity element of $X$ and its group of unit is $\hat{X}$ which is isometrically isomorphic to $X$.
\end{Thm}
The Proposition \ref{int} and Theorem \ref{int2} are in our opinion the arguments showing that the monoid structure of $(Lip^1_+(X),\oplus)$ is in reality a natural extension of the group structure of $(X,.)$ to the set $Lip^1_+(X)$.\\

We use the following result in the proof of Theorem \ref{int2}. This result is the key of this algebraic theory of the inf-convolution. The part $I)\Rightarrow II)$ was proved in \cite{Ba2}. The part $II)\Rightarrow I)$ is new. A more general form in metric space framework not necessarily group is given in section \ref{S1}. 
\begin{Thm}  \label{Fond1} Let $(X,.,d)$ be a group complete metric invariant and let $a\in X$.  Let $f$ and $g$ be two lower semi continuous functions on $(X,d)$. Then, the following assertions are equivalent \\

$I)$ the map $x \mapsto f\oplus g(x)$ has a strong minimum at $a$ \\

$II)$ there exists $(\tilde{y},\tilde{z})\in X\times X$ such that $\tilde{y}  \tilde{z}=a$ and :
 $f$ has a strong minimum at $\tilde{y}$ and $g$ has at strong minimum a $\tilde{z}$. 
\end{Thm}

Theorem \ref{Fond1} also gives the following corollary. Consider the following submonoid of $Lip^1(X)$ 
$$S(X):=\left\{f\in Lip^1(X)/ \hspace{2mm}f  \hspace{2mm} \textnormal{has a strong minimum}\right\}$$
and the metric $\rho$ defined for $f,g\in Lip^1(X)$ by 
\[
\rho(f,g)= \sup_{x\in X} \frac{|f(x)-g(x)|}{1+|f(x)-g(x)|}.
\]
 For a real-valued function $f$ with domain $X$, $\arg\min(f)$ is the set of elements in $X$ that realize the global minimum in $X$,
 $${\arg\min}(f)=\{x\in X:\,f(x)=\inf_{{y\in X}}f(y)\}.$$
 For the class of functions $f\in S(X)$, $ {\arg\min}(f)=\left\{x_f\right\}$ is a singleton, where $x_f$ is the strong minimum of $f$. We identify the singleton  $\left\{x\right\}$ with the element $x$. 
 \begin{Cor} Let $(X,.,d)$ be a group complete metric invariant having $e$ as identity element. Then, $(S(X),\rho)$ is a dense subset of $Lip^1(X)$ and for all $f, g \in S(X)$ we have
\[
{\arg\min}\left(f\oplus g\right)={\arg\min}(f).{\arg\min}(g).
\]
 In other words, the map $ {\arg\min} : (S(X),\oplus, \rho) \rightarrow (X,.,d)$ is continuous monoid morphism and onto. We have the following commutative diagram, where $I$ denotes the identity map on $X$ and $\gamma$ the Kuratowski operator
  \[
  \xymatrix{
     (X,.) \ar[r]^{\gamma} \ar[rd]_{I}  &  (S(X),\oplus) \ar[d]^{{\arg\min}} \\
      & (X,.) } 
 \]
 \end{Cor}

We are also interested on the monoid structure of the set $\mathcal{K}(X)$ of Katetov functions. There are lot of literature on the metric and the topological structure of this space (See for instance \cite{BY}, \cite{KT} and \cite{MJ}). We give in this section some results about the monoid structure of $\mathcal{K}(X)$ when $X$ is a group, and the convex cone structure of the subset $\mathcal{K}_C(X)$ of $\mathcal{K}(X)$ (of convex functions) when $X$ is a Banach space. 
Let $(X,d)$ be a metric space; we say that $f: X\rightarrow \R$ is a 
Katetov map if 
\begin{eqnarray} \label{eq1}
|f(x)-f(y)|\leq d(x,y)\leq f(x)+f(y); \hspace{2mm} \forall x,y\in X.
\end{eqnarray} 
These maps correspond to one-point metric extensions of $X$. We denote by $\mathcal{K}(X)$
the set of all Katetov maps on $X$; we endow it with the sup-metric $$d_{\infty}\left(f,g\right):=\sup_{x\in X}|f(x)-g(x)|<+\infty$$ which turns it
into a complete metric space. Recall that $X$ isometrically embeds in $\mathcal{K}(X)$ via the Kuratowski embedding $\gamma: x\rightarrow
\delta_x$, where $\delta_x(y) := d(x, y)$, and that one has, for any $f \in \mathcal{K}(X)$, that $d_{\infty}(f,\gamma(x)) = f(x)$. It is shown in Section \ref{S2} that $(\mathcal{K}(X),\oplus)$ has a monoid structure and $\mathcal{K}_C(X)$ has a convex cone structure. We obtain the following analogous to the Banach-Stone theorem which say that the metric monoid $(\mathcal{K}(X),\oplus,d_{\infty})$, completely determine the complete metric invariant group $(X,d)$. Note that in the following result, any monoid isometric isomorphism has the canonical form. We do not know if this is the case for other monoids as the set of all convex $1$-Lipschitz bounded from bellow functions defined on Banach space (See Problem 2. in \cite{Ba1}).
 
\begin{Thm} Let $(X,d)$ and $(Y,d')$ be two  complete metric invariant groups. Then, a map $\Phi : (\mathcal{K}(X),\oplus,d_{\infty})\rightarrow (\mathcal{K}(Y),\oplus,d_{\infty})$ is a monoid isometric isomorphism if, and only if there exists a group isometric isomorphism $T: (X,d)\rightarrow (Y,d')$ such that $\Phi(f)=f\circ T^{-1}$ for all $f\in \mathcal{K}(X)$. Consequently, $Aut_{Iso}(\mathcal{K}(X))$ (the group of all isometric automorphism of the monoid $\mathcal{K}(X)$) is isomorphic as group to $Aut_{Iso}(X)$ (the group of all isometric automorphism of the group $X$).
\end{Thm}
%%%%%%%%%%%%%%%%%%%%%%%%%%%%%%%%%%%%%%%%%%%%%%%%%%%%%%%%%%%%%%%%%%
%%%%%%%%%%%%%%%%%%%%%%%%%%%%%%%%%%%%%%%%%%%%%%%%%%%%%%%%%%%%%%%%%%%%%%
%%%%%%%%%%%%%%%%%%%%%%%%%%%%%%%%%%%%%%%%%%%%%%%%%%%%%%%%%%%%%%%%%%%%%%
%%%%%%%%%%%%%%%%%%%%%%%%%%%%%%%%%%%%%%%%%%%%%%%%%%%%%%%%%%%%%%%%%%%%%%%%
\section{The inf-convolution on complete metric space.}\label{S1}
The main theorem of this section (Theorem \ref{Fond0}) extend  [Theorem 3, \cite{Ba2}] and [Corollary 3, \cite{Ba2}] to complete metric invariant space. In [Theorem 3, \cite{Ba2}] and [Corollary 3, \cite{Ba2}], only the part $I)\Rightarrow II)$ was proved in the group context. Here, we give a necessarily and sufficient condition in the more general metric context.\\

We need some notations and definitions. Let $X$ be a set and $\alpha: X\times X \rightarrow X$ be a map. Given $x\in X$ we denote by $\Delta_\alpha(x)$ the following set depending on $\alpha$
$$\Delta_\alpha(x):=\left\{(y,z)\in X\times X: \alpha(y,z)=x\right\}\subset X\times X .$$
Note that $\Delta_\alpha(\alpha(s,t))\neq \emptyset$ for all $s,t\in X$. We also denote by $\Delta_{1,\alpha}(x)$ (respectively, $\Delta_{2,\alpha}(x)$) the projection of $\Delta_\alpha(x)$ on the first (respectively, the second) coordinate: 
$$\Delta_{1,\alpha}(x):=\left\{y\in X/ \exists z_y\in X: \alpha(y,z_y)=x\right\}\subset X.$$
$$\Delta_{2,\alpha}(x):=\left\{z\in X/ \exists y_z\in X: \alpha(y_z,z)=x\right\}\subset X.$$
% \\
%Let $f$ and $g$ be two functions defined on $X$. We define the inf-convolution of $f$ and $g$ with respect the map $\alpha$ as follows 
%$$f\underbrace{\oplus}_{\alpha} g (x):= \inf_{y,z\in X/\alpha(y,z)=x}\left\{f(y)+g(z)\right\}.$$
%%%%%%%%%%%%%%%%%%%%%%%%%%%%%%%%%%%%%%%%%%%%%%%%
%%%%%%%%%%%%%%%%%%%%%%%%%%%%%%%%%%%%%%%%%%%%%%%% 
\begin{Def} \label{def1} Let $(X,d)$ be metric space and $\alpha: X\times X \rightarrow X$, be a map. We say that $\alpha$ is $d$-invariant at $x\in X$, if $\Delta_\alpha(x)\neq \emptyset$ and there exists $L_1, L_2, L'_1, L'_2>0$ such that
$$L_2d(y_1,y_2)\leq d(\alpha(y_1,z),\alpha(y_2,z))\leq L_1d(y_1,y_2); \hspace{2mm}\forall y_1, y_2\in \Delta_{1,\alpha}(x); z\in \Delta_{2,\alpha}(x).$$ 
and
$$L'_2d(z_1,z_2) \leq d(\alpha(y,z_1),\alpha(y,z_2))\leq L'_1d(z_1,z_2); \hspace{2mm}\forall z_1, z_2\in \Delta_{2,\alpha}(x); y\in \Delta_{1,\alpha}(x).$$
\end{Def} 
The set $\Delta_\alpha(x)$ is endowed with the metric induced by the product metric topology of $X\times X$ i.e $\tilde{d}\left((y,z),(y',z')\right):=d(y,y')+d(z,z')$ for all $(y,z),(y',z') \in X\times X$.
\begin{Prop} \label{cont} Let $(X,d)$ be metric space and $\alpha: X\times X \rightarrow X$ be a map. Suppose that $\alpha$ is $d$-invariant at $x\in X$. Then the restriction of $\alpha$ to the set $\Delta_{\alpha}(x)$ is continuous.
\end{Prop}
\begin{proof} Let $(y,z), (y_0,z_0)\in \Delta_{\alpha}(x)$. Then, 
\begin{eqnarray}
d(\alpha(y,z),\alpha(y_0,z_0))&\leq& d(\alpha(y,z),\alpha(y,z_0))+d(\alpha(y,z_0),\alpha(y_0,z_0))\nonumber\\
                              &\leq& L'_1d(z,z_0)+L_1d(y,y_0)\nonumber\\
                              &\leq& \max(L'_1, L_1) \left(d(z,z_0)+d(y,y_0)\right)\nonumber
\end{eqnarray}
This inequality shows that the restriction of $\alpha$ to the set $\Delta_{\alpha}(x)$ is continuous.
\end{proof}

For two functions $f$ and $g$ on $X$, we define the map $\eta_{f,g}$ depending on $f$ and $g$ by
\begin{eqnarray}
\eta_{f,g} : X\times X&\rightarrow& \R\cup \left\{+\infty\right\}\nonumber\\
 (y,z) &\mapsto& f(y)+g(z) \nonumber
\end{eqnarray}

Note that the inf convolution of $f$ and $g$ at $x\in X$, with respect to the law $\alpha$, coincide with the infinimum of $\eta_{f,g}$ on $\Delta_\alpha(x)$
$$f\underbrace{\oplus}_{\alpha} g(x):=\inf_{y,z\in X/\alpha(y,z)=x}\left\{f(y)+g(z)\right\}:=\inf_{(y,z)\in \Delta_\alpha(x)}\eta_{f,g}(y,z).$$
\begin{Exemp} The Definition \ref{def1} is satisfied in the following cases.
\item $1)$ Let $(X,\|.\|)$ be a vector normed space and $\alpha : X\times X\rightarrow X$ be the map defined by $\alpha(y,z):=y+z$. In this case, the inf-convolution correspond to the classical definition of the inf-convolution on vector space and we have $\Delta_{1,\alpha}(x)=\Delta_{2,\alpha}(x)=X$ for all $x\in X$ and $\alpha$ satisfies  $$\|\alpha(y,x)-\alpha(z,x)\|=\|\alpha(x,y)-\alpha(x,z)\|=\|y-z\|; \hspace{2mm}\forall x,y, z\in X.$$
\item $2)$ Let $(C,\|.\|)$ be a convex subset of a vector normed space $(X,\|.\|)$ and let $\lambda\in ]0,1[$ be a fixed real number. Let $\alpha : C\times C\rightarrow C$ be the map defined by $\alpha(y,z):=\lambda y+ (1-\lambda)z$. Then, $\left\{(x,x)\right\}\subset \Delta_\alpha(x)$ and  $\Delta_\alpha(x)= \left\{(x,x)\right\}$ if, and only if $x$ is an extreme point of $C$ and we have 
$$\|\alpha(y,x)-\alpha(z,x)\|=\lambda\|y-z\|; \hspace{2mm}\forall x,y, z\in C.$$
and
$$\|\alpha(x,y)-\alpha(x,z)\|=(1-\lambda)\|y-z\|; \hspace{2mm}\forall x,y, z\in C.$$
\item $3)$ If $(X,.,d)$ is a metric group, ($.$ is the law of internal composition of $X$) and $\alpha : (y,z)\mapsto y. z$,  then $\Delta_{1,\alpha}(x)=\Delta_{2,\alpha}(x)=X$ for all $x\in X$. Moreover, $\alpha$ is $d$-invariant at $x$ for each $x\in X$ if, $(X,.,d)$ is metric invariant. We recall that a metric group is said to be metric invariant, if $d(y. x,z. x)=d(x. y,x. z)= d(y,z)$ for all $x,y,z\in X$. Every group is metric invariant for the discreet metric. We can find examples of group metric invariant in \cite{Ba2}.\\
\item $4)$ However, there exists examples of metric monoids $(M,d)$ with a law $.$ which is not metric invariant but such that $.$ is $d$-invariant at each element of the group of unit of $M$ (See Proposition \ref{x2} and Remark \ref{Z}).
\end{Exemp}
 
\begin{Def} Let $(X,d)$ be a metric space, we say that a function $f$ has a strong minimum at $x_0\in X$, if $\inf_X f= f(x_0)$ and  for all $\epsilon>0$, there exists $\delta>0$ such that $$0\leq f(x)-f(x_0)\leq \delta \Rightarrow d(x,x_0)\leq\epsilon.$$ A strong minimum is in particular unique. By $dom(f)$ we denote the domain of $f$, defined by $dom(f):=\left\{x\in X: f(x)< +\infty \right\}$. All functions in the article are supposed such that $dom(f)\neq \emptyset$.
\end{Def} 

In what follows, an element $\alpha(y,z)\in X$ will simply be noted by $yz$ and the inf-convolution of two functions $f$ and $g$ will simply be denoted by 
$$f\oplus g(x):=\inf_{yz=x}\left\{f(y)+g(z)\right\}.$$
For $a\in X$, we say that $f\oplus g(a)$ is strongly attained at $(y_0,z_0)$, if the restriction of $\eta_{f,g}$ to the set $\Delta_\alpha(a)$ has a strong minimum at $(y_0,z_0)\in \Delta_\alpha(a)$.
%%%%%%%%%%%%%%%%%%%%%%%%%%%%%%%%%%%%%%%%%%%%%%%%%%%%%
%%%%%%%%%%%%%%%%%%%%%%%%%%%%%%%%%%%%%%%%%%%%%%%%%%%%%%
%%%%%%%%%%%%%%%%%%%%%%%%%%%%%%%%%%%%%%%%%%%%%%%%%%%%%
%%%%%%%%%%%%%%%%%%%%%%%%%%%%%%%%%%%%%%%%%
%%%%%%%%%%%%%%%%%%%%%%%%%%%%%%%%%%%%%%%%%
\begin{Thm} \label{Fond0} Let $(X,d)$ be a complete metric spaces. Let $\alpha: X\times X \rightarrow X$, be a map ($\alpha(y,z):= yz$ for all $y,z\in X\times X$). Let $f$ and $g$ be two lower semi continuous functions on $(X,d)$.  Let $a\in X$ and suppose that the map $\alpha$ is $d$-invariant at $a$. Then, the following assertions are equivalent.

\begin{itemize}
\item[$I)$] the map $x \mapsto f\oplus g(x)$ has a strong minimum at $a\in X$

\item[$II)$] there exists $(\tilde{y},\tilde{z})\in \Delta_\alpha(a)$ i.e $\tilde{y}  \tilde{z}=a$, such that :
 $f$ has a strong minimum at $\tilde{y}$ and $g$ has at strong minimum a $\tilde{z}$.
\end{itemize}
Moreover, in this case, we have
\begin{itemize}
\item[$(1)$] the restricted map $\eta_{f,g} : \Delta_\alpha(a)\rightarrow \R\cup \left\{+\infty\right\}\nonumber$ has a strong minimum at $(\tilde{y},\tilde{z})\in \Delta_\alpha(a)$ i.e $f\oplus g(a)$ is strongly attained at $(\tilde{y},\tilde{z})$.
\item[$(2)$] $f(x)-f(\tilde{y})\geq f\oplus g(x \tilde{z})-f\oplus g(a)$ and $g(x)-g(\tilde{z})\geq f\oplus g(\tilde{y}x)-f\oplus g(a)$ for all $x\in X$.
\end{itemize}
\end{Thm}
\begin{proof} First, from the definition of the inf-convolution, for all $y, y', z, z'\in X$, 
\begin{eqnarray}
\label{cool5} f\oplus g(yz') &\leq& f(y)+g(z')\\
\label{cool6} f\oplus g(y'z)&\leq& f(y')+ g(z). 
\end{eqnarray}
 By adding both inequalilies (\ref{cool5}) and (\ref{cool6}) above we obtain 
\begin{eqnarray}\label{cool7}
f\oplus g(yz')+f\oplus g(y'z)&\leq& \left(f(y) +g(z)\right) + \left(f(y')+g(z')\right).
\end{eqnarray}
{\it $I)\Rightarrow II)$.} Replacing $f$ by $f-\frac{1}{2} f\oplus g (a)$ and $g$ by $g-\frac{1}{2} f\oplus g (a)$, we can assume without loss of generality that $f\oplus g$ has a strong minimum at $a$ and $f\oplus g (a)=0$. Let $(y_n)_n;(z_n)_n\subset X$ be such that for all $n\in \N^*$, $y_n z_n=a$ and
\begin{eqnarray} 
0=f\oplus g(a) \leq & f(y_n) + g(z_n) & <f\oplus g(a)+\frac 1 n.\nonumber
 \end{eqnarray}
in other words
 \begin{eqnarray} \label{cool2}
 0\leq & f(y_n) + g(z_n) & < \frac 1 n.
 \end{eqnarray}
By appllaying (\ref{cool7}) with $y=y_n$; $z=z_n$; $y'=y_p$ and $z'=z_p$ we have
 \begin{eqnarray} f\oplus g(y_n z_p)+f\oplus g(y_p z_n) &\leq & \left(f(y_n)+ g(z_n)\right) + \left(f(y_p)+ g(z_p)\right) \nonumber 
 \end{eqnarray} 
Using the above inequality and (\ref{cool2}) we obtain
\begin{eqnarray} \label{cool1.8} 0=2(f\oplus g(a))\leq f\oplus g(y_n z_p)+f\oplus g(y_p z_n) &\leq & \frac{1}{n}+\frac{1}{p}. 
 \end{eqnarray} 
%0\leq d(y_n z_p,y_p z_n)\leq  d(y_n z_p,e_X )+d(y_p z_n,e_X ) = 
Since $x\mapsto f\oplus g$ has a strong minimum at $a$, then  $d(x,a)\rightarrow 0$ whenever $f\oplus g (x)\rightarrow 0$. On the other hand, $f\oplus g(x) \geq f\oplus g(a)=0$ for all $x\in X$. Thus from (\ref{cool1.8}), we get that $f\oplus g(y_n z_p)\rightarrow 0$ and $f\oplus g(y_p z_n)\rightarrow 0$ when $n,p\rightarrow +\infty$. We deduce that $d(y_n z_p,a)\rightarrow 0$, when $n,p\rightarrow +\infty$. Since $y_p z_p=a$ for all $p\in \N$ and $\alpha$ is $d$-invariant at $a$, then $d(y_n,y_p)\leq \frac{1}{L_2} d(y_n z_p,y_pz_p)=\frac{1}{L_2}d(y_n z_p,a)\rightarrow 0$, when $n,p\rightarrow +\infty$. Hence $(y_n)_n$ is a Cauchy sequence and so converges to some point $\tilde{y}\in X$ since $(X,d)$ is complete metric space. Similarly, we prove that $(z_n)_n$ converges to some point $\tilde{z}\in X$. By the continuity of the map $\alpha: (y,z)\mapsto yz$ (See Proposition \ref{cont}), we deduce that $\tilde{y}  \tilde{z}=\lim_n\left(y_n z_n\right)=\lim_n\left(a\right)=a$.\\
 \\
Using the lower semi-continuity of $f$ and $g$ and the formulas (\ref{cool2}) we get 

\begin{eqnarray} f(\tilde{y})+ g(\tilde{z})&\leq& \liminf_{n\rightarrow +\infty} f(y_n)+\liminf_{n\rightarrow +\infty} g(z_n)\nonumber\\
                                    &\leq& \liminf_{n\rightarrow +\infty}\left(f(y_n)+g(z_n)\right)\leq 0=f\oplus g(a).\nonumber
\end{eqnarray}
On the other hand, it is always true that $f\oplus g(a)\leq f(\tilde{y})+g(\tilde{z})$ since $\tilde{y} \tilde{z}=a$. Thus 
\begin{eqnarray} \label{cool8} f(\tilde{y}) + g(\tilde{z}) = f\oplus g(a)=0.
\end{eqnarray}
%It follows that $\inf_{(y,z)\in \Delta\left(X\right)}\eta_{f,g}(y,z)=\eta_{f,g}(\tilde{y},\tilde{z})=f\oplus g(a)=0.$ So, the map $\eta_{f,g} : \Delta_\alpha(a)&\rightarrow& \R\cup \left\{+\infty\right\}\nonumber$ has a minimum at $(\tilde{y},\tilde{z})\in \Delta_\alpha(a)$.\\
%\\
 Using (\ref{cool8}) we obtain
 %have that $0=f\oplus g(a)= f(\tilde{y})+g(\tilde{z})$. Thus
\begin{eqnarray} \label{cool3} f\oplus g(\tilde{y}x)  &\leq& f(\tilde{y})+g(x)= g(x)-g(\tilde{z}). 
\end{eqnarray}
and 
\begin{eqnarray} \label{cool'3} f\oplus g(x\tilde{z})  &\leq& f(x)+g(\tilde{z}) = f(x)-f(\tilde{y}).
\end{eqnarray}
 Using (\ref{cool'3}) and the fact that $f\oplus g$ has a strong minimum at $a$, we have that $f(x)-f(\tilde{y})\geq 0$ and if $f(y_n)-f(\tilde{y})\rightarrow 0$, then $f\oplus g(y_n\tilde{z})\rightarrow 0$ which implies that $d(y_n\tilde{z},a)\rightarrow 0$ since $f\oplus g$ has a strong minimum at $a$. On the other hand we have $d(y_n,\tilde{y})\leq \frac{1}{L_2} d(y_n\tilde{z},\tilde{y}  \tilde{z})= \frac{1}{L_2}d(y_n\tilde{z},a)$ by the $d$-invariance of $\alpha$ at $a$. Thus $d(y_n,\tilde{y})\rightarrow 0$ and so $f$ has a strong minimum at $\tilde{y}$. The same argument, by using (\ref{cool3}), shows that $g$ has a strong minimum at $\tilde{z}$\\
\\
{\it $II)\Rightarrow I)$.}  We first prove that $f\oplus g$ has a minimum at $\tilde{y} \tilde{z}=a$ and that  $f(\tilde{y})+g(\tilde{z})=f\oplus g(a)$. Indeed, since $f(y)\geq f(\tilde{y})$ and $g(z)\geq g(\tilde{z})$ for all $y, z\in X$, then we get $$f\oplus g(a):=\inf_{yz=a}\left\{ f(y)+g(z)\right\}\geq f(\tilde{y})+g(\tilde{z})$$ 
On the other hand, $f\oplus g(a)\leq f(\tilde{y})+g(\tilde{z})$ since $\tilde{y} \tilde{z}=a$. Thus, $f\oplus g(a)= f(\tilde{y})+g(\tilde{z})$. On the other hand, using again the fact that $f(y)\geq f(\tilde{y})$ and $g(z)\geq g(\tilde{z})$ for all $y, z\in X$, we obtain for all $x\in X$, $$f\oplus g(x):=\inf_{yz=x}\left\{ f(y)+g(z)\right\}\geq f(\tilde{y})+g(\tilde{z})=f\oplus g(a).$$
 It follows that $f\oplus g$ has a minimum at $a=\tilde{y}  \tilde{z}$. Now, let $(x_n)_n\subset X$  be a sequence that minimize $f\oplus g$. Let $\epsilon_n\rightarrow 0^+$ such that 
\begin{eqnarray} \label{cip}
f\oplus g (a) \leq f\oplus g (x_n)\leq f\oplus g(a) +\epsilon_n
\end{eqnarray}
 From the definition of $f\oplus g (x_n)$, for each $n\in N^*$, there exists sequences $(y_n)_n, (z_n)_n\subset X$ satisfying $y_nz_n=x_n$ and 

\begin{eqnarray}  f\oplus g(x_n) -\frac{1}{n} \leq f(y_n) +g(z_n) \leq f\oplus g (x_n) +\frac{1}{n}\nonumber
\end{eqnarray}
Since $f(\tilde{y})+g(\tilde{z})=f\oplus g(a)$, it follows that 
\begin{eqnarray*} f\oplus g (x_n) -f\oplus g(a) -\frac{1}{n} & \leq & \left(f(y_n)-f(\tilde{y})\right) +\left(g(z_n)-g(\tilde{z})\right)\\
& \leq & f\oplus g (x_n)-f\oplus g(a) +\frac{1}{n}
\end{eqnarray*}
Using the inequality (\ref{cip}) we get for all $n\in \N^*$
\begin{eqnarray} -\frac{1}{n} \leq \left(f(y_n)-f(\tilde{y})\right) +\left(g(z_n)-g(\tilde{z})\right) \leq \epsilon_n +\frac{1}{n}\nonumber
\end{eqnarray}
Since $\left(f(y_n)-f(\tilde{y})\right)\geq 0$ and $\left(g(z_n)-g(\tilde{z})\right)\geq 0$, we get that 
\begin{eqnarray*} 0 \leq f(y_n)-f(\tilde{y}) &\leq &\left(f(y_n)-f(\tilde{y})\right) +\left(g(z_n)-g(\tilde{z})\right)\\
& \leq &\epsilon_n +\frac{1}{n}
\end{eqnarray*}
and 
\begin{eqnarray} 0 \leq g(z_n)-g(\tilde{z}) \leq \left(f(y_n)-f(\tilde{y})\right) +\left(g(z_n)-g(\tilde{z})\right) \leq \epsilon_n +\frac{1}{n}\nonumber
\end{eqnarray}
Sending $n$ to $+\infty$, we have $\lim_{n\rightarrow +\infty}f(y_n)= f(\tilde{y})$ and $\lim_{n\rightarrow +\infty}g(z_n)= g(\tilde{z})$. Since $f$ and $g$ has respectively a strong minimum at $\tilde{y}$ and $\tilde{z}$, we deduce that $\lim_{n\rightarrow +\infty} y_n=y$ and $\lim_{n\rightarrow +\infty} z_n=z$. By the continuity of the map $\alpha :(y,z)\mapsto yz$ on $\Delta_\alpha(a)$, we obtain $\lim_{n\rightarrow +\infty}  (y_nz_n)=\tilde{y}  \tilde{z}=a$. Since $y_nz_n=x_n$ for all $n\in \N^*$, we have $\lim_{n\rightarrow +\infty} x_n=a.$ Thus $f\oplus g$ has a  strong minimum at $a$\\
\\
Moreover, we have the additional informations:
\begin{itemize}
\item[$(1)$] $f\oplus g(a)$ is strongly attained at $(\tilde{y},\tilde{z})$ (We assume as in $I)$ that $f\oplus g(a)=0$). We know from (\ref{cool8}) that $f\oplus g(a)$ is attained at $(\tilde{y},\tilde{z})$. To see that $\eta_{f,g} $ has in fact a strong minimum at $(\tilde{y},\tilde{z})$, let $((y_n,z_n))_n\subset \Delta_\alpha(a)$ be any sequence such that 
\begin{eqnarray} \label{cool.20}
f(y_n)+g(z_n):=\eta_{f,g}(y_n,z_n)\rightarrow \inf_{yz=a }\left\{f(y)+g(z)\right\}=\eta_{f,g}(\tilde{y},\tilde{z})=0.
\end{eqnarray}
 By applying (\ref{cool7}) with $y=\tilde{y}$, $y'=y_n$, $z=\tilde{z}$ and $z'=z_n$ and the formulas (\ref{cool8}) and (\ref{cool2}),  we obtain 
\begin{eqnarray} \label{cool.9} 0\leq f\oplus g(\tilde{y} z_n)+f\oplus g(y_n \tilde{z}) &\leq& \left(f(\tilde{y}) +g(\tilde{z})\right) + \left(f(y_n)+g(z_n)\right)\nonumber\\
                                                                                             &=& \left(f(y_n)+g(z_n)\right)
\end{eqnarray}
Thus $f\oplus g(\tilde{y} z_n)\rightarrow 0$ (and also $f\oplus g(y_n \tilde{z})\rightarrow 0$) from  (\ref{cool.20}) and (\ref{cool.9}) and the fact that $f\oplus g (x)\geq 0=f\oplus g (a)$, for all $x\in X$. It follows that $d(\tilde{y} z_n,a)\rightarrow 0$ and $d(y_n \tilde{z},a)\rightarrow 0$, since $f\oplus g$ has a strong minimum at $a$. Hence, $d(y_n,\tilde{y})\rightarrow 0$ since $d(y_n,\tilde{y})\leq \frac{1}{L_2}d(y_n \tilde{z},\tilde{y}  \tilde{z})=\frac{1}{L_2}d(y_n \tilde{z},a)$ by the $d$-invariance of $\alpha$ at $a$, and the fact that $\tilde{y} \tilde{z}=a$ . In a similar way we have $d(z_n,\tilde{z})\rightarrow 0$. 
 Thus $(\tilde{y},\tilde{z})$ is a strong minimum of $\eta_{f,g}$.
\item[$(2)$] This part follows from (\ref{cool3}) and (\ref{cool'3})
 \end{itemize}
 \end{proof}
 %%%%%%%%%%%%%%%%%%%%%%%%%%%%%
 %%%%%%%%%%%%%%%%%%%%%%%%%%%%%
 \section{The monoids structure for the inf-convolution.}
 The following corollary will permit to describe the group of unit of submonoids, for the inf-convolution structure, of the set $Lip^1(X)$ of all $1$-Lipschitz and bounded from below functions.
 \begin{Cor} \label{lips} Let $(X,.,d)$ be a group complete metric invariant having $e$ as identity element. Let $f$ and $g$ be two $1$-Lipschitz functions on $X$. Then, the following assertions are equivalent. 
 \begin{itemize}
 \item[$(1)$] $f\oplus g =d(e,.)$.
 \item[$(2)$] there exists $\tilde{y}\in X$ and $c\in \R$ such that  
 $$f(.)=d(\tilde{y},.)+c:=\gamma(\tilde{y})+c$$
 and 
 $$g(.)=d(\tilde{y}^{-1},.)-c=\gamma(\tilde{y}^{-1})-c.$$
 \end{itemize}
 \end{Cor}
\begin{proof} $\bullet (1)\Rightarrow (2)$. Since the law $.$ is in particular $d$-invariant at $e$ and the map $d(e,.)$ has a strong minimum at $e$, by applying Theorem \ref{Fond0}, there exists $\tilde{y}, \tilde{z}\in X$ such that $\tilde{y}\tilde{z}=e$, $f(\tilde{y}) +g(\tilde{z})= f\oplus g(e)=0$ and $f(x)-f(\tilde{y})\geq d(e,x\tilde{z})=d(\tilde{y},x)$ and $g(x)-g(\tilde{z})\geq d(e,\tilde{y}x)=d(\tilde{y}^{-1},x)$, for all $x\in X.$ On the other hand, since $f$ and $g$ are $1$-Lipschitz, we have $f(x)-f(\tilde{y})\leq d(\tilde{y},x)$ and $g(x)-g(\tilde{z})\leq d(\tilde{z},x)=d(\tilde{y}^{-1},x)$ , for all $x\in X.$ Thus, $f(x)-f(\tilde{y})=d(\tilde{y},x)$ and $g(x)-g(\tilde{y}^{-1})=g(x)-g(\tilde{z})=d(\tilde{y}^{-1},x)$ , for all $x\in X.$ In other words, $f(.)=d(\tilde{y},.)+c:=\gamma(\tilde{y})+c$ and $g(.)=d(\tilde{y}^{-1},.)-c=\gamma(\tilde{y}^{-1})-c$, with $c=f(\tilde{y})=-g(\tilde{z})$.

$\bullet (2)\Rightarrow (1)$. Suppose that $(2)$ hold, then  $f\oplus g(x)=\left(\gamma(\tilde{y})\oplus \gamma(\tilde{y}^{-1})\right)(x)$ for all $x\in X.$ Since $(X,.,d)$ is group complete metric invariant, by using Proposition \ref{int}. we get $f\oplus g=\gamma(e):=d(e,.)$
\end{proof}

\begin{Lem} \label{loop} Let $(X,.,d)$ be a metric space and $. : (y,z)\mapsto yz$ be a law of composition of $X$.
\begin{itemize} 
\item[$1)$] Suppose that $d(yx,zx)\leq d(y,z)$ and $d(xy,xz)\leq d(y,z)$, for all $x, y, z\in X$. Then we have, for all $a,b \in X$  $$\gamma(ab)\leq \gamma(a)\oplus \gamma(b).$$
\item[$2)$] Suppose $d(yx,zx)=d(xy,xz)=d(y,z)$, for all $x, y, z\in X$. Then, we have for all $x\in X$ and all $a,b \in X$ 
$$\left(\gamma(a)\oplus\gamma(b)\right)(xb)= \gamma(ab)(xb)$$
and 
$$\left(\gamma(a)\oplus\gamma(b)\right)(ax)=\gamma(ab)(ax).$$
If moreover $X$ is quasigroup, then we have for all $a,b \in X$ $$\gamma(a)\oplus\gamma(b)= \gamma(ab).$$
\end{itemize}
\end{Lem}
\begin{proof}
$1)$ Let $a,b , x \in X$, then we have
\begin{eqnarray}
\gamma(a)\oplus\gamma(b)(x):=\inf_{yz=x}\left\{d(a,y) +d(b,z)\right\} &\geq &\inf_{yz=x}\left\{d(az,yz)+d(ab,az)\right\}  \nonumber\\
                                               &\geq &\inf_{yz=x} d(ab,yz) \nonumber\\
                                               &= & d(ab,x):=\gamma(ab)(x)\nonumber
\end{eqnarray}
$2)$ Using the metric invariance, we have for all $x\in X$ and all $a,b \in X$
\begin{eqnarray} \left(\gamma(a)\oplus\gamma(b)\right)(xb) &:=&\inf_{yz=xb}\left\{d(a,y) +d(b,z)\right\}\nonumber\\
                                                          &\leq & d(a,x)\nonumber\\
                                                          &=& d(ab,xb)\nonumber\\
                                                          &:=& \gamma(ab)(xb)\nonumber
\end{eqnarray}
Combining this inequality with the part $1)$, we get $\left(\gamma(a)\oplus\gamma(b)\right)(xb)=\gamma(ab)(xb)$. In a similar way, we prove $\left(\gamma(a)\oplus\gamma(b)\right)(ax)=\gamma(ab)(ax)$. If moreover, $X$ is a quasigroup, then for each $t, b\in X$, there exists $x\in X$ such that $t=xb$. So we obtain $\gamma(a)\oplus\gamma(b)= \gamma(ab)$
\end{proof}

We give  the proof of Proposition \ref{int} mentioned in the introduction.

\begin{proof}[{\bf Proof of Proposition \ref{int}.}] $(1)\Rightarrow (2)$. Suppose that $(X,\alpha)$ is quasigroup. Using Lemma \ref{loop}, we have that $\gamma(a)\oplus \gamma(b)=\gamma(ab)$, for all $a,b \in X$. Using this formula and the injectivity of $\gamma$ , it is clear that $(\gamma(X), \oplus)$ is quasigroup (respectively, loop, group, commutative group) whenever $(X,.)$ is quasigroup (respectively, loop, group, commutative group).\\
The last part of the theorem, follows from the formula $\gamma(a)\oplus \gamma(b)=\gamma(ab)$ and the fact that $\gamma$ is isometric.\\
$(1)\Rightarrow (2)$.  Suppose that $(\gamma(X), \oplus)$ is a quasigroup. Let us prove that $(X,.)$ is quasigroup. First, we show that for all $a,b \in X$, we have that $\gamma(a)\oplus \gamma(b)=\gamma(ab)$. Indeed, let $a, b\in X$. Since $\oplus$ is an internal law of $(\gamma(X), \oplus)$, then there exists $c\in X$ such that $\gamma(a)\oplus \gamma(b)=\gamma(c)$. Using Lemma \ref{loop}, we obtain $\gamma(ab)\leq\gamma(c)$. Hence $0\leq d(ab,c)=\gamma(ab)(c)\leq\gamma(c)(c)=0$. This implies that $c=ab$. Finally we have $\gamma(a)\oplus \gamma(b)=\gamma(ab)$ for $a, b\in X$. From this formula and the injectivity of $\gamma$, it is clear that $(X,.)$ is quasigroup (respectively, loop, group, commutative group) whenever $(\gamma(X), \oplus)$ is quasigroup (respectively, loop, group, commutative group).
\end{proof}

The following Corollary is a particular case of the work established in \cite{Ba2}.

\begin{Cor} \label{3} Let $(X,.,d)$ be a group complete metric invariant having $e$ as identity element. Then,
\begin{itemize}
\item[$(1)$] the set $Lip^1(X)$ of all $1$-Lipschitz and bounded from below functions, is a monoid having $\gamma(e):=d(e,.)$ as identity element and its group of unit $\mathcal{U}(Lip^1(X))$ coincides with $\hat{X}+\R$. 
\item[$(2)$] the set $Lip^1_+(X)$ of all $1$-Lipschitz and positive functions, is a monoid having $\gamma(e):=d(e,.)$ as identity element and its group of unit $\mathcal{U}(Lip^1_+(X))$ coincides with $\hat{X}$.
\end{itemize}
 \end{Cor}

\begin{proof}

$(1)$ The fact that $Lip^1(X)$ is a monoid having $\gamma(e):=d(e,.)$ as identity element, follows from Proposition 7. and Lemma 3. in \cite{Ba2}. Using Proposition \ref{int}, we have that $\hat{X}+\R \subset \mathcal{U}(Lip^1(X))$. The fact that $\mathcal{U}(Lip^1(X))\subset \hat{X}+\R$, follows from Corollary \ref{lips}. Thus $\mathcal{U}(Lip^1(X))= \hat{X}+\R$.

$(2)$ Since the inf-convolution of positive functions is also positive and $\gamma(e):=d(e,.)\in Lip^1_+(X)$, then $Lip^1_+(X)$ is a submonoid of $Lip^1(X)$. On the other hand, $\hat{X}\subset \mathcal{U}(Lip^1_+(X))\subset \mathcal{U}(Lip^1(X))=\hat{X}+\R$. Since the element of $\mathcal{U}(Lip^1_+(X))$ are positive functions we get $\mathcal{U}(Lip^1_+(X))=\hat{X}$.
\end{proof}

We give now the proof of Theorem \ref{int2} mentioned in the introduction.

\begin{proof}[{\bf Proof of Theorem \ref{int2}.}] The part  $(1)\Rightarrow (2)$ can be deduced from \cite{Ba2} (See also \cite{Ba1}). Let us prove $(2)\Rightarrow (1)$. Since $(Lip^1_+(X),\oplus)$ is a (commutative) monoid, there exists and identity element $f_0\in Lip^1_+(X)$. Since $f_0$ is the identity element, it satisfies in particular:  $\gamma(a)\oplus f_0=f_0\oplus\gamma(a)=\gamma(a)$ for all $a\in X$. Since  $\gamma(a):=d(a,.)$ has a strong minimum at $a$, applying Theorem \ref{Fond0} to the functions $\gamma(a)$ and $f_0$ and $\gamma(a)\oplus f_0$, there exists $(\tilde{y},\tilde{z})\in X\times X$ satisfying $\tilde{y}  \tilde{z}=a$ such that $\gamma(a)$ has a strong minimum at $\tilde{y}$ and $f_0$ has a strong minimum at $\tilde{z}$. Since a strong minimum is in particular unique, then $\tilde{y}=a$. So, we have $a \tilde{z}=a$, for all $a\in X$. In a similar way we prove that $\tilde{z} a=a$ , for all $a\in X$. Thus $e:=\tilde{z}$ is the identity element of $X$. From the associativity of $(Lip^1_+(X),\oplus)$, we obtain in particular the associativity of $(\hat{X},\oplus)$. Since $(X,.)$ is a quasigroup then from Lemma \ref{loop}, we have $\gamma(a)\oplus \gamma(b)= \gamma(ab)$ for all $a,b\in X$, so we deduce by the injectivity of $\gamma$, that $(X,.)$ is also associative. Hence, $(X,.)$ is a (commutative) group. The fact that, the identity element of  $(Lip^1_+(X),\oplus)$ is $\gamma(e)$ where $e$ is the identity element of $X$ and its group of unit is $\hat{X}$, follows from Corllary \ref{3}.
\end{proof}
  
%%%%%%%%%%%%%%%%%%%%%%%%%%%%%%%%%
%%%%%%%%%%%%%%%%%%%%%%%%%%%%%%%%%%
\section{Metric properties and the density of $S(X)$ in $Lip^1(X)$.}
Let us consider the following sets $$S(X):=\left\{f\in Lip^1(X)/ \hspace{2mm}f  \hspace{2mm} \textnormal{has a strong minimum}\right\}$$  
$$S_+(X):=\left\{f\in Lip^1_+(X)/ \hspace{2mm}f  \hspace{2mm} \textnormal{has a strong minimum}\right\}$$
 \begin{Cor} Let $(X,.,d)$ be a group complete metric invariant having $e$ as identity element. Then, $S(X)$ is a submonoid of $(Lip^1(X),\oplus)$ and $\mathcal{U}(S(X))=\hat{X}+\R$. On the other hand
 $S_+(X)$ is a submonoid of $(Lip^1_+(X),\oplus)$ and $\mathcal{U}(S_+(X))=\hat{X}$.
 \end{Cor}
 
\begin{proof} 

Since $(Lip^1(X),\oplus)$ is a monoid having $\gamma(e)\in S(X)$ as identity element and since $S(X)$ is a subset of $(Lip^1_+(X),\oplus)$, it suffices to show that $\oplus$ is an internal law of $S(X)$ which is the case thanks to Theorem \ref{Fond0}. On the other hand, $\hat{X}+\R\subset \mathcal{U}(S(X))\subset \mathcal{U}(Lip^1(X))=\hat{X}+\R$. Hence $\mathcal{U}(S(X))=\hat{X}+\R$. In a similar way we obtain the second part of the Corollary.
\end{proof}

Consider now the metrics $\rho$ and $\tilde{\rho}$ on  $Lip^1(X)$ defined for $f,g\in Lip^1(X)$ by 
$$\rho(f,g)= \sup_{x\in X} \frac{|f(x)-g(x)|}{1+|f(x)-g(x)|}$$
$$\tilde{\rho}(f,g)= \rho(f-\inf_X f,g-\inf_X g) +|\inf_X f- \inf_X g|.$$
%Note that $\rho$ and $\tilde{\rho}$ coincides on 
\begin{Prop} Let $(X,d)$ be a complete metric space. Then,
\begin{itemize}
\item[$(1)$] the sets $(Lip^1(X),\rho)$ and $(Lip^1(X),\tilde{\rho})$ (respectively, $(Lip^1_+(X),\rho)$ and $(Lip^1_+(X),\tilde{\rho})$) are complete metric spaces.
\item[$(2)$] the set $(S(X),\rho)$ is dense in $(Lip^1(X),\rho)$ and $(S(X),\tilde{\rho})$ is dense in $(Lip^1(X),\tilde{\rho})$.
\item[$(3)$] the set $(S_+(X),\rho)$ is dense in $(Lip^1_+(X),\rho)$ and $(S_+(X),\tilde{\rho})$ is dense in $(Lip^1_+(X),\tilde{\rho})$.
\end{itemize}
\end{Prop}

\begin{proof}
The part $(1)$ is similar to Proposition 5. in \cite{Ba1} (See also Lemma 1. in \cite{Ba1}). Let us prove the part $(2)$. Indeed, let $f\in Lip^1(X)$ and $0<\epsilon<1$. Consider the function $f_\epsilon:=(1-\epsilon) f$. Clearly, $f_\epsilon$ is $(1-\epsilon)$-Lipschitz and $\rho(f_\epsilon,f)\rightarrow 0$ (respectively $\tilde{\rho}(f_\epsilon,f)\rightarrow 0$) when $\epsilon \rightarrow 0$. On the other hand, applying the variational principle of Deville-Godefroy-Zizler \cite{DGZ} to the $(1-\epsilon)$-Lipschitz and bounded from below function $f_\epsilon$, there exists a bounded Lipschitz function $\varphi_\epsilon$ on $X$ such that $\sup_{x\in X} |\varphi_\epsilon(x)| \leq \epsilon$ and $\sup_{x, y \in X/ x\neq y} \frac{|\varphi_\epsilon(x)-\varphi_\epsilon(y)|}{|x-y|}\leq \epsilon$ and $f_\epsilon+\varphi_\epsilon$ has a strong minimum at some point. We have that $f_\epsilon+\varphi_\epsilon$ is $1$-Lipschitz and bounded from bellow function having a strong minimum, so $f_\epsilon+\varphi_\epsilon \in S(X)$. On the other hand, $\rho(f_\epsilon+\varphi_\epsilon,f)\leq \rho(f_\epsilon+\varphi_\epsilon,f_\epsilon) + \rho(f_\epsilon,f)=\rho(\varphi_\epsilon,0)+\rho(f_\epsilon,f)$. It follows that $\rho(f_\epsilon+\varphi_\epsilon,f)\rightarrow 0$ when $\epsilon \rightarrow 0$ (respectively $\tilde{\rho}(f_\epsilon+\varphi_\epsilon,f)\rightarrow 0$). Thus $(S(X),\rho)$ and $(S(X),\tilde{\rho})$ are respectively dense in $(Lip^1(X),\rho)$ and $(Lip^1(X),\tilde{\rho})$.

To prove $(3)$, let $f\in Lip^1_+(X)$, from $(2)$, there exists $f_\epsilon \in S(X)$ such that $\rho(f_\epsilon,f)\rightarrow 0$ when $\epsilon \rightarrow 0$. In particular, $\inf_X f_\epsilon \rightarrow \inf_X f$. If $\inf_X f >0$, then for very small $\epsilon$ we have $f_\epsilon >0$ and so $f_\epsilon\in S_+(X)$. If $\inf_X f=0$, since $\inf_X f_\epsilon \rightarrow \inf_X f=0$ then  $\rho(f_\epsilon-\inf_X f_\epsilon,f)\leq \rho(f_\epsilon-\inf_X f_\epsilon,f_\epsilon)+\rho(f_\epsilon,f)\rightarrow 0$ when $\epsilon \rightarrow 0$ and $\left(f_\epsilon-\inf_X f_\epsilon\right)\in S_+(X)$. Thus, $(S_+(X),\rho)$ is dense in $(Lip^1_+(X),\rho)$. We deduce then that $(S_+(X),\tilde{\rho})$ is also dense in $(Lip^1_+(X),\tilde{\rho})$
\end{proof}
%%%%%%%%%%%%%%%%%%%%%%%%%%%%%%%%%%%
%%%%%%%%%%%%%%%%%%%%%%%%%%%%%%%%%%%%
\section{The map $\arg\min(.)$ as monoid morphism.}

 For a real-valued function $f$ with domain $X$, $\arg\min(f)$ is the set of elements in $X$ that realize the global minimum in $X$,
 $${\arg\min}(f)=\{x\in X:\,f(x)=\inf_{{y\in X}}f(y)\}.$$
 For the class of functions $f\in S(X)$, $ {\arg\min}(f)=\left\{x_f\right\}$ is a singleton, where $x_f$ is the strong minimum of $f$. In what follows, we identify the singleton  $\left\{x\right\}$ with the element $x$. We have the following proposition.
 \begin{Cor} Let $(X,.,d)$ be a group complete metric invariant having $e$ as identity element. Then, 
  the map, $$ {\arg\min} : (S(X),\oplus,\rho) \rightarrow (X,.,d)$$ is surjective and continuous monoid morphism. We have the following commutative diagram, where $I$ denotes the identity map on $X$ and $\gamma$ the Kuratowski operator
  \[
  \xymatrix{
     (X,.) \ar[r]^{\gamma} \ar[rd]_{I}  &  (S(X),\oplus) \ar[d]^{{\arg\min}} \\
      & (X,.) } 
 \]
 \end{Cor}
 
\begin{proof} Let $f,g \in S(X)$, then there exists $x_f, x_g\in X$ such that $f$ has a strong minimum at $x_f={\arg\min}(f)$ and $g$ has a strong minimum at $x_g={\arg\min}(g)$. Using Theorem \ref{Fond0}, $f\oplus g$ has a strong minimum at $x_f x_g=\left({\arg\min}(f)\right)\left({\arg\min}(g)\right)$. Thus ${\arg\min}(f\oplus g)=\left({\arg\min}(f)\right)\left({\arg\min}(g)\right)$. On the other hand, the map ${\arg\min}$ send the identity element $d(e,.)$ of $S(X)$ to the identity element $e$ of $X$, since the strong minimum of $d(e,.)$ is $e$. Hence, ${\arg\min}$ is a monoid morphism. For each $x\in X$, $\gamma(x)\in \hat{X}\subset S(X)$ and  ${\arg\min}\left( \gamma(x)\right)=x$. Thus, ${\arg\min}$ is surjective. Let us prove now the continuity of ${\arg\min}$. First, note that for all $f, g\in Lip^1(X)$ and all $0<\alpha<1$, 
\begin{eqnarray} \label{cont1} \rho\left(f,g\right)\leq \alpha \Rightarrow \sup_{x\in X}|f(x)-g(x)|\leq \frac{\alpha}{1-\alpha}.
\end{eqnarray}
and in consequence, we also have 
\begin{eqnarray} \label{cont2} |\inf_X f-\inf_X g|\leq \frac{\alpha}{1-\alpha}.
\end{eqnarray}
Let $(f_n)_n\subset S(X)$ and $f\in S(X)$. Let $x_n:={\arg\min}(f_n)$ and $x_f={\arg\min}(f)$. Since $f$ has a strong minimum at $x_f$, for all $\epsilon>0$, there exists $\delta>0 $ such that for all $x\in X$,
$$|f(x)-f(x_f)|\leq \delta \Rightarrow d(x,x_f)\leq \epsilon.$$
Suppose that $\rho\left(f_n,f\right)\leq \frac{\delta}{2+\delta}$. Using the triangular inequality and the inequations (\ref{cont1}) and (\ref{cont2}) with $\alpha=\frac{\delta}{2+\delta}<1$, we have
\begin{eqnarray} |f(x_n)-f(x_f)| &\leq& |f(x_n)-f_n(x_n)|+|f_n(x_n)-f(x_f)|\nonumber\\
                              &=&  |f(x_n)-f_n(x_n)|+|\inf_X f_n-\inf_X f|\nonumber\\
                              &\leq& \frac{2\alpha}{1-\alpha}\nonumber\\
                              &=& \delta \nonumber
\end{eqnarray}
which implies that $d({\arg\min}(f_n), {\arg\min}(f)):=d(x_n,x_f)\leq \epsilon$. This implies the continuity of ${\arg\min}$ on $S(X)$
\end{proof}

Note that, the map $\xi : (Lip^1(X),\oplus,\rho) \rightarrow \R$ defined by $\xi: f\mapsto \inf_X f$, is continuous monoid morphism. 
%We do not know to describe other continuous monoid morphism from $(Lip^1(X),\oplus,\rho)$ into $\R$. 
The following proposition which is a consequence of the above corollary, says that, in the case of $(S(X),\oplus,\rho)$ there are several continuous monoid morphism from $(S(X),\oplus)$ into $\K$ with $\K=\R$ or $\C$.
\begin{Prop} Let $(X,.,d)$ be a group complete metric invariant and $\chi : (X,.,d)\rightarrow \K$ be a continuous group morphism. Then, $\chi\circ {\arg\min}: (S(X),\oplus,\rho) \rightarrow \K$ is a continuous monoid morphism.
\end{Prop}
 %%%%%%%%%%%%%%%%%%%%%%%%%%%%%%%%%%%%%%%%%%%%%%%%%%%%
 %%%%%%%%%%%%%%%%%%%%%%%%%%%%%%%%%%%%%%%%%%%%%%%%%%%%
 %%%%%%%%%%%%%%%%%%%%%%%%%%%%%%%%%%%%%%%%%%%%%%%%%%%%
\section{Examples of inf-convolution monoid in the discrete case.} 
Let $X$ be a group with the identity element $e$. We equip $X$ with the discrete metric $dis$. So $(X,dis)$ is group metric invariant and $Lip^1_+(X)$ consist in this case on all positive functions such that $|f(x)-f(y)|\leq 1$ for all $x,y\in X$. The Kuratowski operator $\gamma : x\mapsto \delta_x$ is here defined by $\delta_x(y)=1$ if $y\neq x$ and $\delta_x(x)=0$, for all $x, y\in X$. We treat below the cases of $X=\Z$ and $X=\Z/p\Z$.
\subsection{The inf-convolution monoid $(l^{\infty}_{dis}(\Z),\oplus)$.} 
Let $X=\Z$ equipped with the discrete metric $dis$ which is invariant.
Let $l_{dis}^{\infty}(\Z)$ the set of all sequences $(x_n)_n$ of real positive numbers such that $|x_n-x_m|\leq 1$ for all $n,m\in \Z$. For $u=(u_n)_n$ and $v=(u_n)_n$ in $l^{\infty}_{dis}(\Z)$, we define the sequence $$(u\oplus v)_n:=\inf_{k\in \Z} (u_{n-k}+v_k); \hspace{3mm} \forall n\in \Z.$$
The set $(l^{\infty}_{dis}(\Z),\oplus)$ is a commutative monoid having the element $\delta_e$ as identity element and its group of unit $\mathcal{U}(l^{\infty}_{dis}(\Z))$ is isomorphic to $\Z$ by the isomorphism $I: \Z\rightarrow \mathcal{U}(l^{\infty}_{dis}(\Z))$, $k\mapsto \delta_k$. 
\subsection{The inf-convolution monoid $(l^{\infty}_{dis}(\Z/p\Z),\oplus)$.} 
Let $p\in \N^*$ and $X=\Z/p\Z$ equipped with the discrete metric $dis$ which is invariant. We denote by $l^{\infty}_{dis}(\Z/p\Z)$ the set of all $p$-periodic sequences $(x_n)_n$ of real positive numbers such that $|x_n-x_m|\leq 1$ for all $n,m\in \left\{0,...,p-1\right\}$. We identify a sequence $(x_n)_n\in l^{\infty}_{dis}(\Z/p\Z)$ with $\left(x_0,...,x_{p-1}\right)$. For $u=(u_n)_n$ and $v=(u_n)_n$ in $l^{\infty}_{dis}(\Z/p\Z)$, we define the sequence $$(u\oplus v)_n:=\min_{k\in \left\{0,...,p-1\right\}} (u_{n-k}+v_k); \hspace{3mm} \forall n\in \left\{0,...,p-1\right\}.$$
The set $(l^{\infty}_{dis}(\Z/p\Z),\oplus)$ is a commutative monoid having the element $\delta_e$ as identity element and its group of unit $\mathcal{U}(l^{\infty}_{dis}(\Z/p\Z))$ is isomorphic to $\Z/p\Z$ by the isomorphism $I: \Z/p\Z\rightarrow \mathcal{U}(l^{\infty}_{dis}(\Z/p\Z))$, $\bar{k}\mapsto \delta_k$.
%%%%%%%%%%%%%%%%%%%%%%%%%%%%%%%%%%%%%%%%%%%%
%%%%%%%%%%%%%%%%%%%%%%%%%%%%%%%%%%%%%%%%%%%%
%%%%%%%%%%%%%%%%%%%%%%%%%%%%%%%%%%%%%%%%%%%%
\section{The set of Katetov functions.}\label{S2}
 We give in this section some results about the monoid structure of $\mathcal{K}(X)$ when $X$ is a group, and the convex cone structure of the subset $\mathcal{K}_C(X)$ of $\mathcal{K}(X)$ (of convex functions) when $X$ is a Banach space. If $M$ is a monoid, by $\mathcal{U}(M)$ we denote the group of unit of $M$.
%%%%%%%%%%%%%%%%%%%%%%%%%%%%%%%%%
%%%%%%%%%%%%%%%%%%%%%%%%%%%%%%%%%%
%\section{The group structure.}
\subsection{The monoid structure of $\mathcal{K}(X)$.}
\begin{Prop} \label{x2} Let $(X,d)$ be a (commutative) group metric invariant having $e$ as identity element. 
Then, the metric space $(\mathcal{K}(X),\oplus, d_{\infty})$ is also a (commutative) monoid having $\gamma(e)=\delta_e$ as identity element and  satisfying:
\item $(a)$  $d_{\infty}(f\oplus g,h\oplus g)\leq d_{\infty}(f,h)$ and $d_{\infty}(g\oplus f,g\oplus h)\leq d_{\infty}(f,h)$, for all $f,g, h\in \mathcal{K}(X)$ 
\item $(b)$  $d_{\infty}(\delta_x\oplus f,\delta_x\oplus h)= d_{\infty}(f\oplus \delta_x,h\oplus \delta_x)=d_{\infty}(f,h)$, for all $f,h\in \mathcal{K}(X)$.
%In particular, $(\mathcal{K}(X),\oplus,d_{\infty})$ satisfies also the part (ii) of property (P).
\end{Prop}

\begin{proof} Since $\mathcal{K}(X)$ is a subset of the (commutative) monoid $Lip^1(X)$ of $1$-Lipschitz and bounded from below functions, which have $\delta_e$ as identity element (See \cite{Ba2}), it suffices to prove that, for all $f,g\in \mathcal{K}(X)$ and all $x_1,x_2\in X$, we have $$d(x_1,x_2) \leq f\oplus g(x_1)+f\oplus g(x_2)$$ 
 Indeed, it follows easily from the definition of the infinimum, the formula (\ref{eq1}) and the metric invariance that, for all $n\in \N^*$, there exists $y_n, z_n, y'_n, z'_n\in X$ such that $y_nz_n=x_1$,  $y'_n z'_n=x_2$ and
\begin{eqnarray} f\oplus g(x_1)+f\oplus g(x_2) &\geq& \left(f(y_n)+g(z_n)+\frac1 n\right)+\left(f(y_n')+g(z_n')+\frac1 n\right) \nonumber \\
                                               &=& \left(f(y_n)+f(y_n')\right)+\left(g(z_n)+g(z_n')\right)+\frac2 n\nonumber \\
                                               &\geq& d(y_n,y_n')+d(z_n,z_n') +\frac2 n\nonumber \\
                                               &= & d(y_nz_n,y_n'z_n)+d(y'_nz_n,y'_nz'_n)+\frac2 n \nonumber \\
                                               &=& d(x_1,y_n'z_n)+d(y'_nz_n,x_2)+\frac2 n\nonumber \\
                                               &\geq& d(x_1,x_2)+\frac2 n\nonumber                                              
\end{eqnarray}
Thus $f\oplus g(x_1)+f\oplus g(x_2)\geq d(x_1,x_2)$ by sending $n$ to $+\infty$. Hence $(\mathcal{K}(X),\oplus)$ is a monoid having $\delta_e$ as identity element.\\
We prove now that $d_{\infty}(f\oplus g,h\oplus g)\leq d_{\infty}(f,h)$. Let $f,g,h\in \mathcal{K}(X)$ and $x\in X$, there exists $y_n, z_n$ such that $y_nz_n=x$ and $h\oplus g(x)>h(y_n)+g(z_n) -\frac1 n$. Hence, for all $n\in \N^*$ 
\begin{eqnarray} f\oplus g(x)-h\oplus g(x) &\leq& \left(f(y_n)+g(z_n)\right)+\left(-h(y_n)-g(z_n) +\frac1 n\right)\nonumber \\
                                               &=& f(y_n)-h(y_n)+\frac1 n\nonumber \\
                                               &\leq& d_{\infty}(f,h)+\frac1 n\nonumber 
\end{eqnarray} 
Hence, $d_{\infty}(f\oplus g,h\oplus g)\leq d_{\infty}(f,h)$ by sending $n$ to $+\infty$. In a similar way we prove that $d_{\infty}(g\oplus f,g\oplus h)\leq d_{\infty}(f,h)$. For the part $(b)$, it suffices to prove that $f\oplus \delta_a(.)=f(.a^{-1})$ and $\delta_a \oplus f (.)=f(a^{-1}.)$ for all $a\in X$, since the map  $x\mapsto ax$ and $x\mapsto xa$ are one to one and onto from $X$ to $X$ whenever $a$ is invertible. Indeed, $f\oplus \delta_a(x)=\inf_{yz=x}\left\{f(y)+d(z,a)\right\}=\inf_{(ya^{-1})(az)=x} \left\{f(ya^{-1})+d(az,a)\right\}$. Using the metric invariance, we have for all $x\in X$,
$f\oplus \delta_a(x)=\inf_{yz=x}\left\{f(ya^{-1})+d(z,e)\right\}:=f(.a^{-1})\oplus \delta_e (x)=f(.a^{-1})(x)$, since $\delta_e$ is the identity element. Similarly we prove that $\delta_a \oplus f (.)=f(a^{-1}.)$. This conclude the proof of the proposition.
\end{proof}

\begin{Rem} \label{Z} In general, one can not get equality in the part $(a)$ of Proposition \ref{x2} since the inf-convolution does not have the cancellation property in general (See \cite{ZA}).
\end{Rem}

If $ Y \subset X$ and $f\in \mathcal{K}(Y)$, define  $\overline{f} : X \rightarrow \R$ (the Katetov extension of f)
by $f(x) = inf_{y\in Y}\left\{f(y) + d(x, y)\right\}$. It is well known that $\overline{f}$ is the greatest 1-Lipschitz map on $X $ which is equal to $f$ on $Y$ ; that $\overline{f}\in \mathcal{K}(X)$ and $\chi: f \mapsto \overline{f}$ is an isometric
embedding of $\mathcal{K}(Y)$ into $\mathcal{K}(X)$ (see for instance \cite{KT}). Thanks to the following lemma (we can find a more general form in \cite{Ba2}) we can assume without loss of generality that $X$ is a complete metric space.

%%%%%%%%%%%%%%%%%%%%%%%%%%%%%%%%%%%%%%%%%%%%%%
%%%%%%%%%%%%%%%%%%%%%%%%%%%%%%%%%%%%%%%%%%%%%%%%
%%%%%%%%%%%%%%%%%%%%%%%%%%%%%%%%%%%%%%%%%%%%%%%%%%%
\begin{Lem} \label{Extension} Let $(X,d)$ be a group which is metric invariant and $(\overline{X},\overline{d})$ its group completion. Then, $(\mathcal{K}(X),\oplus,d_{\infty})$ and $(\mathcal{K}(\overline{X}),\oplus,d_{\infty})$ are isometrically isomorphic as monoids. More precisely, the map
\begin{eqnarray}
\chi: (\mathcal{K}(X),\oplus,d_{\infty})&\rightarrow& (\mathcal{K}(\overline{X}),\oplus,d_{\infty})\nonumber\\
 f&\mapsto& \overline{f}:=\left[\overline{x}\in \overline{X} \mapsto \inf_{y\in X}\left\{f(y)+\overline{d}(y,\overline{x})\right\}\right]
\end{eqnarray}
is an isometric isomorphism of monoids.
  \end{Lem}
\begin{proof}
It suffices to shows that $\chi$ is a surjective morphism of monoids. The surjectivity is clear, since if $F\in \mathcal{K}(\overline{X})$, we take  $f=F_{|X}$, then $\overline{f}=F$ on $X$ and so $\overline{f}=F$ on $\overline{X}$ by continuity.  Let us show that $\chi$ is a monoid morphism. Indeed, let $f,g \in \mathcal{K}(X)$. Using the continuity of $\overline{f}$, $\overline{g}$ and $y\mapsto y^{-1}$ and the density of $X$ in $\overline{X}$, we have for all $x\in X$,
\begin{eqnarray}
\overline{f}\oplus \overline{g}(x)&=&\inf_{\tilde{y}\hspace{0.3mm}\tilde{z}=x}\left\{\overline{f}(\tilde{y})+\overline{g}(\tilde{z}))\right\}\nonumber\\
&=& \inf_{\tilde{y}\in\overline{X}}\left\{\overline{f}(\tilde{y})+\overline{g}(\tilde{y}^{-1}x)\right\}\nonumber\\ 
&=& \inf_{y\in X}\left\{f(y)+g(y^{-1}x)\right\}.\nonumber\\ 
&=& f\oplus g (x)\nonumber
\end{eqnarray}
 Thus  $\overline{f}\oplus \overline{g}$ coincides with $f\oplus g=\overline{f\oplus g}$ on $X$. Hence $\overline{f}\oplus \overline{g}=\overline{f\oplus g}.$
\end{proof}

 %%%%%%%%%%%%%%%%%%%%%%%%%%%%%%%%%%%%%%%%
 %%%%%%%%%%%%%%%%%%%%%%%%%%%%%%%%%%%%%%%%%%
 %%%%%%%%%%%%%%%%%%%%%%%%%%%%%%%%%%%%%%%%%%
%%%%%%%%%%%%%%%%%%%%%%%%%%%%%%%%%%%%%%%%%%
%%%%%%%%%%%%%%%%%%%%%%%%%%%%%%%%%%%%%%%%%
%In the following corollary, both parts (i) and (ii) of property (P) are crucial.
The following theorem shows that, up to an isometric isomorphism of groups, the group of unit of $\mathcal{K}(X)$ and the group of unit of $X$ are the same.
\begin{Prop} \label{inv} Let $(X,d)$ be a group which is complete metric invariant. Then, the group of unit $\mathcal{U}(\mathcal{K}(X))$ and $\hat{X}$ (which is isometrically isomorphic to $X$) coincides.  
\end{Prop}

\begin{proof} Since $(X,d)$ be a complete metric invariant group, from Proposition \ref{int} we get that $\hat{X}\subset \mathcal{U}(\mathcal{K}(X))$. On the other hand,  $\mathcal{U}(\mathcal{K}(X))\subset \mathcal{U}(Lip^1_+(X))=\hat{X}$.
\end{proof}
%%%%%%%%%%%%%%%%%%%%%%%%%%%%%%%%%%%%%%%%
%%%%%%%%%%%%%%%%%%%%%%%%%%%%%%%%%%%%%%%%

We deduce the following analogous to the Banach-Stone theorem.
\begin{Thm} \label{coucou} Let $(X,d)$ and $(Y,d')$ be two  complete metric invariant groups. Then, a map $\Phi : (\mathcal{K}(X),\oplus,d_{\infty})\rightarrow (\mathcal{K}(Y),\oplus,d_{\infty})$ is a monoid isometric isomorphism if, and only if there exists a group isometric isomorphism $T: (X,d)\rightarrow (Y,d')$ such that $\Phi(f)=f\circ T^{-1}$ for all $f\in \mathcal{K}(X)$. Consequently, $Aut_{Iso}(\mathcal{K}(X))$  is isomorphic as group to $Aut_{Iso}(X)$.
\end{Thm}

\begin{proof} If $T: (X,d)\rightarrow (Y,d')$ is an group isometric isomorphism, clearly $\Phi(f):=f\circ T^{-1}$ gives an monoid isometric isomorphism  from $(\mathcal{K}(X),\oplus,d_{\infty})$ onto $(\mathcal{K}(Y),\oplus,d_{\infty})$.

For the converse, let $\Phi$ be monoid isometric isomorphism  from $(\mathcal{K}(X),d_{\infty})$ onto $(\mathcal{K}(Y),d_{\infty})$, then $\Phi$ maps isometrically the group of unit of $\mathcal{K}(X)$ onto the group of unit of $\mathcal{K}(Y)$. Using Proposition \ref{inv}, $\Phi$ maps isometrically the group $\hat{X}$ onto $\hat{Y}$. Then, the map $$T:=\gamma^{-1}\circ \Phi_{|\hat{X}}\circ\gamma$$ gives an isometric group isomorphism from $X$ onto $Y$ by Proposition \ref{int}, where $\Phi_{|\hat{X}}$ denotes the restriction of $\Phi$ to $\hat{X}$. Since $\Phi$ is isometric we have for all $f\in \mathcal{K}(X)$ and all $x\in X$
$$f(x)=d_{\infty}\left(f,\delta_x\right)=d_{\infty}\left(\Phi(f),\Phi(\delta_x)\right)=d_{\infty}\left(\Phi(f),\delta_{T(x)}\right)=\Phi(f)\left(T(x)\right)$$
which conclude the proof.
\end{proof}
%%%%%%%%%%%%%%%%%%%%%%%%%%%%%%%%%%%%%%%%%
%%%%%%%%%%%%%%%%%%%%%%%%%%%%%%%%%%%%%%%%%
\begin{Lem} \label{x3} Let $(M,d)$ be a metric monoid, $\mathcal{U}(M)$ its group of unit. Suppose that $d(xu,yu)\leq d(x,y)$ and $d(ux,uy)\leq d(x,y)$ for all $x, y,u\in M$, and $d(xu,yu)=d(ux,uy)= d(x,y)$ for all $x, y\in M$ and all $u\in\mathcal{U}(M)$. Then, for all $x \in X$ and all $a,b\in \mathcal{U}(X)$, we have the following formula
$$d(x,ab)=\inf_{yz=x}\left\{d(y,a) +d(z,b)\right\}.$$ 
\end{Lem}

\begin{proof} Let  $a, b \in \mathcal{U}(M)$ and $x\in M$, we have
\begin{eqnarray}
\inf_{yz=x}\left\{d(y,a) +d(z,b)\right\} &\leq& d(xb^{-1},a)\hspace{2mm} (\textnormal{with}\hspace{2mm} y=xb^{-1}; z=b)\nonumber\\
                                                            &=& d(x,ab)\nonumber                                                          
\end{eqnarray}
On the other hand, 
\begin{eqnarray}
\inf_{yz=x}\left\{d(y,a) +d(z,b)\right\} &\geq &\inf_{yz=x}\left\{d(yz,az)+d(az,ab)\right\}  \nonumber\\
                                               &\geq &\inf_{yz=x} d(yz,ab) \hspace{2mm} (\textnormal{by using the triangular inequality})\nonumber\\
                                               &= & d(x,ab).\nonumber
\end{eqnarray}
Thus, $d(x,ab)=\inf_{yz=x}\left\{d(y,a) +d(z,b)\right\}$.
\end{proof}

We obtain the following formula.
\begin{Cor} \label{coucou1} Let $(X,d)$ be a group which is metric invariant.  Let $f\in \mathcal{K}(X)$ and $a, b\in X$. Then
$$f(ab)=\inf_{\varphi\oplus \psi=f}\left\{\varphi(a)+\psi(b)\right\}.
$$
\end{Cor}
\begin{proof} Since the monoid  $(\mathcal{K}(X),\oplus, d_{\infty})$ satisfy the Proposition \ref{x2}, by applying Lemma \ref{x3} to the monoid $M=(\mathcal{K}(X),\oplus, d_{\infty})$ and using the fact that $\hat{X}\subset \mathcal{U}(\mathcal{K}(X))$($= \hat{\overline{X}}$) and $d_{\infty}(g,\gamma(x))=g(x)$ for all $x\in X$ and all $g\in \mathcal{K}(X)$, we obtain for all $f\in \mathcal{K}(X)$ and all $a, b\in X$: 
\begin{eqnarray}
f(ab)=d_{\infty}(f,\gamma(ab))=d_{\infty}(f,\gamma(a)\oplus \gamma(b))&=&\inf_{\varphi\oplus \psi=f}\left\{d_{\infty}(\varphi,\gamma(a)) + d_{\infty}(\psi,\gamma(b))\right\}\nonumber\\
& =& \inf_{\varphi\oplus \psi=f}\left\{ \varphi(a)+ \psi(b))\right\}.\nonumber
\end{eqnarray}
This conclude the proof.
\end{proof}
%%%%%%%%%%%%%%%%%%%%%%%%%%%%%%%%
%%%%%%%%%%%%%%%%%%%%%%%%%%%%%%%
%%%%%%%%%%%%%%%%%%%%%%%%%%%%%%
\subsection{The convex cone structure of $\mathcal{K}_C(X)$.}
%%%%%%%%%%%%%%%%%%%%%%%%%%%%%%%
%%%%%%%%%%%%%%%%%%%%%%%%%%%%%%%%
%%%%%%%%%%%%%%%%%%%%%%%%%%%%%%%%
Let $(X,\|.\|)$ be a Banach space. We recall that $\mathcal{K}_C(X):=\left\{f\in \mathcal{K}(X): f \hspace{2mm} \textnormal{convex} \right\}.$ Since the inf-convolution of convex functions is convex, the set $(\mathcal{K}_C(X),\oplus)$ is a complete metric space and commutative submonoid of $(\mathcal{K}(X),\oplus)$. We equip $\mathcal{K}_C(X)$ with the external law $\star$ defined as follows: for all $f\in \mathcal{K}_C(X)$ and all $\lambda\in \R^{+}$ by
$$\lambda \star f\left(x\right):=\lambda f\left(\frac{x}{\lambda}\right); \forall x\in X\hskip2mm if \hskip2mm \lambda>0$$
$$0\star f:=\gamma(0):=\|.\|.$$
%\[\lambda . f\left(x\right):=
%\left\{
%\begin{array}{rl}
%\lambda f\left(\frac{x}{\lambda}\right); \forall x\in X\hskip2mm if \hskip2mm \lambda>0\\
%\gamma(0)(x):=\|x\|; \forall x\in X\hskip2mm if \hskip2mm \lambda=0.
%\end{array}
%\right.\star 
%\]
 We recall below the definition of a convex cone.
\begin{Def}\label{cone}
A commutative monoid $(C,\oplus)$ equipped with a scalar multiplication map
\begin{eqnarray} \star : \R^+\times C & \rightarrow & C\nonumber \\
(\lambda,c) & \mapsto & \lambda\star c \nonumber 
\end{eqnarray} is said to be a convex cone if and only if it satisfies the following properties :
\begin{itemize}
\item[$1)$] $1\star c=c$ and $0\star c=e_C$, for all $c\in C$ where $e_C$ denotes the identity element of $(C,\oplus)$.
\item[$2)$] $\left(\alpha+\beta\right)\star c=\left(\alpha\star c\right)\oplus\left(\beta\star c\right)$ for all $\alpha, \beta \in \R^+$ and all $c\in C$. 
\item[$3)$] $\lambda\star\left(c\oplus c'\right)=\left(\lambda\star c\right) \oplus \left(\lambda\star c'\right)$ for all $\lambda \in \R^+$ and all $c,c'\in C$.
\end{itemize} 
\end{Def}
The following proposition is easily verified.
\begin{Prop} The space $(\mathcal{K}_C(X),\oplus,\star,d_{\infty})$ is a complete metric convex cone with the identity element $\gamma(0)$.
\end{Prop}
The complete metric convex cone structure of $(\mathcal{K}_C(X),\oplus,\star,d_{\infty})$ induce a structure of Banach space on $\hat{X}$ by setting $\lambda\star\gamma(x):=(-\lambda)\star\gamma(-x)$, if $\lambda <0$ and taking the norm $|||\gamma(x)|||:=d_{\infty}\left(\gamma(x),\gamma(0)\right)$ for all $x\in X$. In fact, we can also say that the Banach space $X$ extend its structure canonically to some convex cone structure on $\mathcal{K}_C(X)$.
\begin{Prop} The Kuratowski operator  $\gamma : \left(X,+,.,\|.\|\right)\rightarrow \left(\hat{X},\oplus,\star,|||.|||\right)$ is an isometric isomorphism of Banach spaces. 
\end{Prop}

\begin{proof} Using Proposition \ref{int}, it just remains to prove that $\gamma(\lambda x)=\lambda\star\gamma(x)$ for all $x\in X$ and $\lambda\in \R$. Indeeed, let $x\in X$ and $\lambda\in \R^{*+}$, by definition $\gamma(\lambda x):y\mapsto \delta_{\lambda x}(y)=\|y-\lambda x\|=\lambda\|\frac{y}{\lambda}-x\|:=\lambda\star\delta_x(y)$. So, $\gamma(\lambda x)=\lambda\star\gamma(x)$. If $\lambda=0$, then by definition $0\star\gamma(x)=\gamma(0)$. If $\lambda<0$, $\gamma(\lambda x)=\gamma((-\lambda)(-x))=(-\lambda)\star\gamma(-x):=\lambda\star\gamma(x)$.
\end{proof}
\begin{Thm} Let $X$ and $Y$ two Banach spaces. Then, the spaces $(\mathcal{K}_C(X),\oplus,\star,d_{\infty})$ and  $(\mathcal{K}_C(Y),\oplus,\star,d_{\infty})$ are isometrically isomorphic as convex cone if, and only if, $X$ and $Y$ are isometrically isomorphic as Banach spaces.
\end{Thm}

\begin{proof} Similar to the proof of Theorem \ref{coucou}.
\end{proof}

Using the fixed point Theorem, we obtain the following proposition.
\begin{Prop} Let $X$ be a Banach space, $g\in \mathcal{K}_C(X)$ and $\lambda \in (0,1)$. Then, there exists a unique function $f_0\in \mathcal{K}_C(X)$ such that $(\lambda \star f_0)\oplus g=f_0$.
\end{Prop}

\begin{proof} Let us consider the map $L: \mathcal{K}_C(X)\rightarrow \mathcal{K}_C(X)$ defined by $L(f)=(\lambda \star f)\oplus g$. Using Proposition \ref{x2} we have for all $f,f'\in \mathcal{K}_C(X)$, $$d_{\infty}(L(f),L(f'))\leq d_{\infty}(\lambda \star f,\lambda \star f')=\lambda d_{\infty}(f,f').$$
Since $\lambda \in (0,1)$, then $L$ is contractive map. So by the fixed point Theorem, there exists a unique function $f_0\in \mathcal{K}_C(X)$ such that $L(f_0)=f_0$.
\end{proof}
%\end{Thm}
\bibliographystyle{amsplain}

\end{document}